\documentclass[12pt,leqno]{amsart}
\usepackage[utf8]{inputenc}
\usepackage[top=1in, bottom=1in, left=1in, right=1in]{geometry}
\usepackage[english]{babel}
\usepackage{float}

\usepackage{subcaption}
\usepackage{amsmath}
\usepackage{amssymb}
\usepackage{amsfonts}
\usepackage{amsthm}
\usepackage{mathtools}
\usepackage{mathrsfs}
\usepackage[labelfont=bf,labelsep=period,justification=raggedright]{caption}
\usepackage{xcolor}

\usepackage{amsmath,amssymb,amsthm,graphicx,polynom} 

\newtheorem{theorem}{Theorem}[section]

\newtheorem{lemma}[theorem]{Lemma}
\newtheorem{proposition}[theorem]{Proposition}
\newtheorem{definition}[theorem]{Definition}

\newtheorem{remark}[theorem]{Remark}

\theoremstyle{definition}
\newtheorem{exmp}[theorem]{Example}

\title{Reciprocity and the Kernel of Dedekind Sums}
\author{Alexis LaBelle}
\address{High Point University \\ High Point, NC 27268}
\email{alabelle@highpoint.edu}

\author{Emily Van Bergeyk}
\address{Azusa Pacific University \\ Azusa, CA 91702}
\email{evanbergeyk17@apu.edu}

\author{Matthew P. Young}
\address{Department of Mathematics \\
 	  Texas A\&M University \\
 	  College Station \\
 	  TX 77843-3368 \\
 		U.S.A.}

 \email{myoung@math.tamu.edu}

\begin{document}

\begin{abstract} We use the action of Atkin-Lehner operators to generate a family of reciprocity formulas for newform Dedekind sums. 
This family of reciprocity formulas provides symmetries which we use to investigate the kernel of these Dedekind sums.
\end{abstract}

\maketitle

\section{Introduction} \label{intro}
\subsection{Motivation} \label{motive}
 Dedekind sums were introduced as a means of expressing the transformation formula of the Dedekind eta function. 
These sums appear in numerous contexts, including topology, quadratic reciprocity, and modular forms. 
 More background information on the classical Dedekind sum and Dedekind eta function can be found in \cite{apostle}. 
 
 Newform Dedekind sums $S_{\chi_1, \chi_2}$ associated to a pair of primitive Dirichlet characters $\chi_1, \chi_2$ (modulo $q_1$, $q_2$, respectively) were defined in \cite{2019}, and many of their basic properties were developed therein.  
 The most important property is that each Dedekind sum is a group homomorphism on $\Gamma_1(q_1 q_2)$ (see Lemma \ref{grouphomom}).
 In addition, there is a reciprocity formula relating $S_{\chi_1, \chi_2}$ to $S_{\chi_2, \chi_1}$  (see \cite[Thm 1.3]{2019} or \eqref{eq:SVYreciprocity} below for the precise statement).
 This reciprocity formula generalizes the well-known reciprocity formula for the classical Dedekind sum, which has numerous applications including a proof of quadratic reciprocity, fast calculation of Dedekind sums, etc.
 The main result in this paper, stated in Theorem \ref{IntroSimplifiedReciprocity} below, further generalizes this by giving a \emph{family} of reciprocity formulas.  

The kernel of a Dedekind sum, i.e., the set of elements of $\Gamma_1(q_1 q_2)$ for which the Dedekind sum vanishes, was introduced and studied in \cite{2020}.  The reciprocity formula from \cite{2019} played a crucial role in \cite{2020} in understanding certain experimentally-observed patterns in the kernel.
However, there were other observed patterns that were not explained in \cite{2020}.
As an application of our new reciprocity formulas, our second main result, stated in Theorem \ref{kerthm1} below, identifies additional families of elements lying in the kernel.

The key tool in the proof of the reciprocity formula in \cite{2019} is understanding the action of the Fricke involution on Dedekind sums.  Here we extend this by developing the action of all the Atkin-Lehner operators on the newform Dedekind sums.

\subsection{The reciprocity formula}
 We begin with some notation.  Let $\chi_1 \pmod{q_1}$ and $\chi_2 \pmod{q_2}$ be primitive Dirichlet characters, with $\chi_1 \chi_2(-1) = 1$.  
Let $B_1$ be the first Bernoulli function defined by $$B_1(x)=
\begin{cases}
    x-\left \lfloor{x}\right \rfloor-\frac{1}{2}, & \text{if } x \in \mathbb{R} \backslash \mathbb{Z} \\
    0, &  \text{if } x \in \mathbb{Z}. 
\end{cases}$$ 
Suppose $\gamma = (\begin{smallmatrix}
a & b \\
c & d 
\end{smallmatrix}) \in \Gamma_0(q_1 q_2)$ with $c \geq 1$.
 For $q_1, q_2 > 1$, define the Dedekind sum by
\begin{equation}\label{dede}
S_{\chi_1,\chi_2}(\gamma)= S_{\chi_1, \chi_2}(a,c) = \sum\limits_{j\hspace{-0.6em}\mod c}\sum\limits_{n\hspace{-0.6em}\mod q_1}\overline{\chi_2}(j)\overline{\chi_1}(n) B_1\left(\frac{j}{c}\right)B_1\left(\frac{n}{q_1}+\frac{aj}{c}\right).
\end{equation}
To be clear, we will use a different definition of the Dedekind sum (see Section \ref{dedesums}) which is better-suited for our purposes.  In fact, \eqref{dede} is derived in \cite{2019} only after substantial work. 

 Let $W_Q$ be an Atkin-Lehner operator on $\Gamma_0(N)$ as defined in Section \ref{atkin}, with $N= q_1 q_2 = QR$, and with $(Q, R) = 1$.
A Dirichlet character $\chi_i \pmod{q_i}$ factors uniquely as $\chi_i = \chi_i^{(Q)} \chi_i^{(R)}$ where $\chi_i^{(Q)}$ has conductor $(q_i, Q) =: q_i^{(Q)}$ and  $\chi_i ^ {(R)}$ has conductor $(q_i, R) =: q_i^{(R)}$. Define $\chi_1', \chi_2'$ by
\begin{equation}\label{defprimes}
\chi_1'=\chi_2^{(Q)}\chi_1^{(R)}
\qquad
\text{and}
\qquad
\chi_2'=\chi_1^{(Q)}\chi_2^{(R)}.
\end{equation}
In words, $\chi_1'$ and $\chi_2'$ result from exchanging the $Q$-portions of $\chi_1$ and $\chi_2$.
Here $\chi_i'$ has conductor $q_i'$, where $q_1'=q_2^{(Q)}q_1^{(R)}$ and $q_2'=q_1^{(Q)}q_2^{(R)}$.

In simplified cases, our new family of reciprocity formulas takes the following form.
\begin{theorem} [Simplified Reciprocity Formula]\label{IntroSimplifiedReciprocity} 
Let $q_1,q_2$ be such that $q_1 q_2 > 1$.  
Let $N = q_1 q_2$, suppose $\gamma \in \Gamma_1 (N)$, and define $\gamma' \in \Gamma_1 (N)$ by $W_Q \gamma = \gamma' W_Q$.
Then
\begin{equation}
\label{eq:ALreciprocitySimplified}
 S_{\chi_1,\chi_2}(\gamma') = {\xi}S_{\chi'_1, \chi'_2}(\gamma) 
\end{equation}
\noindent where $\xi$ has absolute value one and depends on $\chi_1$, $\chi_2$, $Q$, and the entries of $W_Q$.
\end{theorem}
\noindent The more general statement of this reciprocity formula, including an explicit value of $\xi$, is presented in Theorem \ref{rec}.

A special case of Theorem \ref{IntroSimplifiedReciprocity} occurs with 
$ W_N=(\begin{smallmatrix}
0 & -1 \\
q_1q_2 & 0 
\end{smallmatrix})$ 
chosen to be the Fricke involution. If $\gamma = (\begin{smallmatrix} a & b \\ c q_1 q_2 & d \end{smallmatrix}) \in \Gamma_1(q_1 q_2)$, then $\gamma' = (\begin{smallmatrix} d & -c \\ -b q_1 q_2 & a \end{smallmatrix})$. Moreover, \eqref{defprimes} says $\chi_1' = \chi_2$ and $\chi_2' = \chi_1$, and  \eqref{eq:ALreciprocitySimplified} becomes
\begin{equation}
\label{eq:SVYreciprocity}
    S_{\chi_1, \chi_2}(\gamma) = \pm S_{\chi_2, \chi_1}(\gamma')
\end{equation}
where $\pm = +1$ if both $\chi_i$ are even, and $\pm = -1$ if both $\chi_i$ are odd. This is essentially Theorem 1.3 in \cite{2019}.

Note that \cite{2019} and \cite{2020} considered Dedekind sums with $q_1>1$ and $q_2 >1$. 
It can happen in \eqref{eq:ALreciprocitySimplified} that $q_1>1$ and $q_2>1$, but that $q_1' =1$ or $q_2'=1$.  For this reason, we were naturally led to extend the definition of Dedekind sums to allow at most one $q_i$ to equal $1$, and to develop some of their properties generalizing results from \cite{2019} and \cite{2020}.  The reader should beware that the formula \eqref{dede} no longer holds when some $q_i=1$, and in fact $S_{1,\chi_2}(\gamma)$ does not depend only on the first column of $\gamma$ (e.g., see \eqref{eq:DedSumUpperTriangularchi1trivial} below).  
Some of the forthcoming results, such as Theorem \ref{kerthm1} below, have some hypotheses ultimately due to subtleties arising from $q_i =1$.

\subsection{The kernel of Dedekind sums} \label{introkernel}

\begin{definition}[\cite{2020} Defn. 1.4] 
Let $\chi_1$ and $\chi_2$ be primitive Dirichlet characters modulo $q_1$ and $q_2$, respectively, with $q_1>1$ and $q_2 > 1$. Then let 
\begin{equation*}
K_{\chi_1,\chi_2}=  \{\gamma \in \Gamma_0(N) \mid S_{\chi_1,\chi_2}(\gamma) = 0\},
\quad
 K_{q_1,q_2}= \cap_{\chi_1, \chi_2} K_{\chi_1,\chi_2},
 \quad
 K_{q_1,q_2}^1= K_{q_1,q_2} \cap \Gamma_1(N).
\end{equation*}
\end{definition} 
An investigation of these kernels was undertaken in \cite{2020}.  The authors used SageMath \cite{sage} and \eqref{dede} to compute $S_{\chi_1, \chi_2}(a,c)$ for all $a \pmod{c}$ and $c \leq 10 q_1 q_2$, and various small values of $q_1, q_2$.  Using this data, they recorded those elements lying in $K_{\chi_1, \chi_2}$, $K_{q_1, q_2}$, and so on.
In Figure \ref{figureIntro2020elements}, we have reproduced the elements of $K_{3,5}$ from \cite{2020},
where the horizontal axis corresponds to the $a$-value, and the vertical axis corresponds to the $c$-value.
In the graph of $K_{3,5}$ in  Figure \ref{figureIntro2020elements}, we have highlighted elements in the kernel that were proved to exist in \cite{2020} and which crucially used \eqref{eq:SVYreciprocity}. 
By applying Theorem \ref{rec}, we 
additionally 
explain all but two of the remaining elements in Figure \ref{figureIntro2020elements} (see Theorem \ref{kerthm1} and \eqref{-1}).

\begin{figure}[H]
  \centering
  \includegraphics[width=3in]{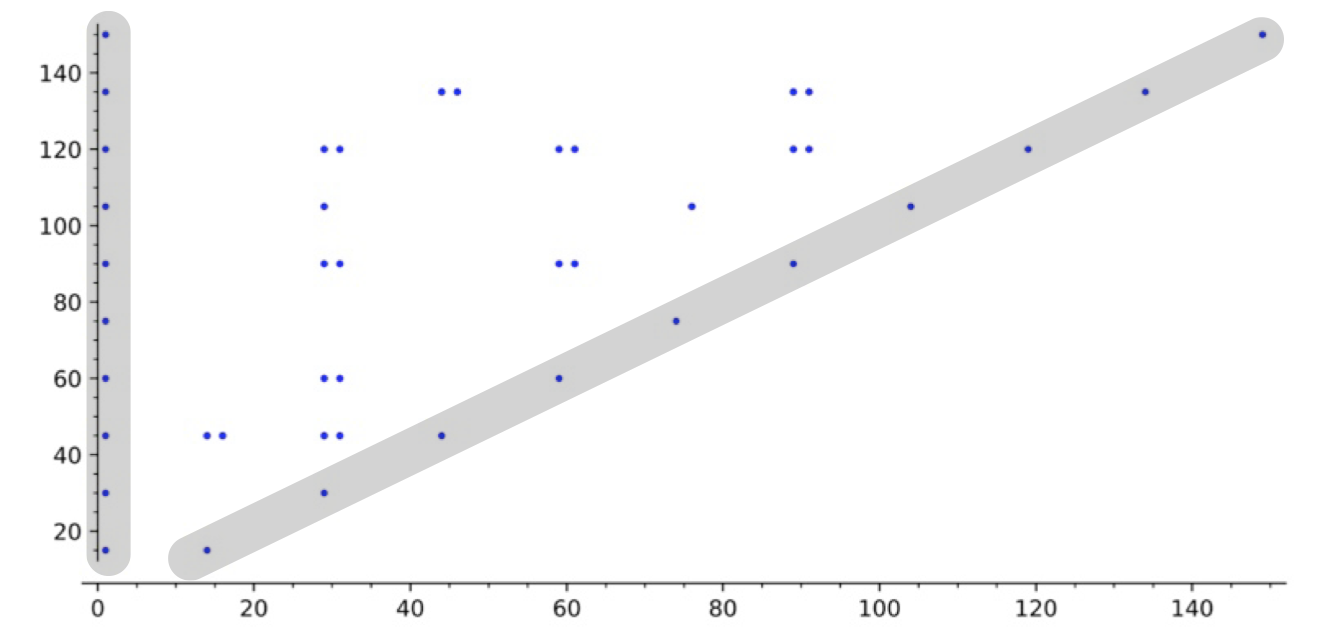}
  \caption{Elements of $K_{3,5}$ with $1 \leq c \leq 10q_1q_2$}
  \label{figureIntro2020elements}
\end{figure}

 The following relationship follows from Theorem \ref{IntroSimplifiedReciprocity}: 
 \begin{equation}\label{iff}
 S_{\chi_1', \chi_2'}(\gamma) = 0 \iff S_{\chi_1, \chi_2}(\gamma') = 0.
 \end{equation}
 Technically, the simplified version in Theorem \ref{IntroSimplifiedReciprocity} gives 
 this relation if $\gamma, \gamma' \in \Gamma_1(N)$, but the more general reciprocity formula in Theorem \ref{rec} implies \eqref{iff} for $\gamma, \gamma' \in \Gamma_0(N)$.
 Using \eqref{iff}, we derive the following theorem (see Section \ref{kernel} for the proof).
\begin{theorem}
\label{kerthm1}
Let $q_1 q_2 = N=QR$ with $(Q,R)=1$, $q_1 \neq 1$, and $q_2 \neq R$. Suppose $r,u \in \mathbb{Z}$ with $(r,R)=1$ and $(u,Q)=1$.  Then for any $k \in \mathbb{Z}$, we have 
$$S_{\chi_1,\chi_2}(\pm 1 + Nkur,NRku^2)=0. $$
\end{theorem} 
\begin{exmp} 
We illustrate Theorem \ref{kerthm1}'s ability to justify many new elements displayed in Figure \ref{figureIntro2020elements}.
Let $N=15$, $R = 3$, $k=u=1$.  Taking $r=1,2$ reveals that $(\pm 1 + 15,45)$ and $(\pm 1 + 30,45)$ are in the kernel.
Taking $k=2,3$ demonstrates the same for the non-highlighted points in Figure \ref{figureIntro2020elements} with $c=90$ and $c=135$.
Now let $R = 1$. Then, taking $r=1$, $u=2$, and $k=1,2$ explains that the non-highlighted points with $c=60$ and $c=120$ are also in the kernel. The only remaining unexplained points from Figure \ref{figureIntro2020elements} are those such that $c=105$ and $a=29$, $76$. These points are not obtainable from Theorem \ref{kerthm1}. This can be seen without computation by noting that the points proved to be in the kernel via Theorem \ref{kerthm1} come in pairs with the values of $a$ separated by two, yet the points in the row $c=105$ do not occur in such pairs.x



Also note that since $a^2 \equiv 1 \pmod{105}$ for $a=29, 76$, then \cite[Prop. 2.2]{2020} may be used to explain these remaining points; we leave the details for an interested reader.

\end{exmp}

\section{Background}

\subsection{Atkin-Lehner Operators}\label{atkin}
Let $N = QR$ with $(Q,R) = 1$.  Following \cite{ali}, we define
an Atkin-Lehner operator $W_Q$ on $\Gamma_0(N)$ as follows.  
Let $r_0, u_0 \in \mathbb{Z}$ be such that $(r_0, R) = 1$ and $(u_0, Q) = 1$.  Let
 $r,t,u,v \in \mathbb{Z}$ with $r \equiv r_0 \pmod{R} $, $u \equiv u_0 \pmod{Q}$, and $Qrv-Rut=1$.  Then we define
$$W_{Q}^{(r_0, u_0)} = W_Q = \begin{pmatrix}
Qr & t \\
Nu & Qv 
\end{pmatrix}.$$ 
This definition preserves the essential properties (see Lemma \ref{a-lprop} below) of the operators as given in \cite{ali}, which took $u_0 = r_0 = 1$. The added flexibility in $u_0$ and $r_0$ will be helpful in Section \ref{kernel}. The relaxed restrictions also mean the Fricke involution is a specialization of the Atkin-Lehner operators: take $Q = N$, $R=1$, $u = u_0 =1$, $t= -1$, $r=v=r_0=0$, so $W_Q =  
(\begin{smallmatrix}
0 & -1\\
N & 0
\end{smallmatrix})$, the Fricke involution.

Suppose that $W_Q$ and $W_Q'$ are Atkin-Lehner operators (with possibly different values of $r,t,u,v$, but the same choice of $r_0$, $u_0$).  For $\gamma \in \Gamma_0(N)$, define $\gamma'$ via
\begin{equation} \label{wqgamma}
W_Q \gamma = \gamma'W_Q'.
\end{equation}

\begin{lemma}
\label{a-lprop}
With $\gamma'$ defined as in \eqref{wqgamma}, we have $\gamma' \in \Gamma_0(N)$.  Let $d_{\gamma}$ and $d_{\gamma'}$ be the lower-right entries of $\gamma$ and $\gamma'$, respectively.  Then
\begin{equation}
    d_{\gamma'} \equiv \begin{cases}
        d_{\gamma} \pmod{R}, \\
        d_{\gamma}^{-1} \pmod{Q}.
    \end{cases}
\end{equation}
\end{lemma}
\begin{remark}
\label{remark:AtkinLehnerRelation}
\normalfont By taking $\gamma = (\begin{smallmatrix} 1 & 0 \\ 0 & 1 \end{smallmatrix})$, Lemma \ref{a-lprop} gives that $W_Q = \alpha W_Q'$ for some $\alpha \in \Gamma_1(N)$.
\end{remark}
\begin{proof}
Say $\gamma = (\begin{smallmatrix} a & b \\ c N & d \end{smallmatrix}) \in \Gamma_0(N)$.
By brute force, we compute $\gamma' = W_Q \gamma W_Q'^{-1}$ as
\begin{equation}
\label{eq:bruteforce}
\gamma' = 
\begin{pmatrix}
Qrv'a+ Ntv'c-Nru'b-Rtu'd & -rt'a- Rtt'c+Qrr'b+tr'd \\
N(uv'a+Qvv'c-Ruu'b-vu'd) & -Rut'a-Nvt'c+Nur'b+Qvr'd
\end{pmatrix}.
\end{equation}
Since $\det(W_Q) = \det(W_Q') =Q$, then
 $\det(\gamma') = 1$.  By inspection, $\gamma'$ has integer entries, and its lower-left
entry is $\equiv 0 \pmod N$. Modulo $R$, the lower-right entry is 
$$d_{\gamma}' \equiv Qvr'd_{\gamma} 
\equiv 
(Qvr - Rut) d_{\gamma} \equiv d_{\gamma} \pmod{R},$$
where we have used $r' \equiv r_0 \equiv r \pmod{R}$.
Similarly, using $a \equiv d_{\gamma}^{-1} \pmod{N}$, we have
$$d_{\gamma}' \equiv -R u t' a \equiv (Qvr - Rut) d_{\gamma}^{-1} \equiv d_{\gamma}^{-1} \pmod{Q}. \qedhere$$
\end{proof}

The Atkin-Lehner operators act on functions as follows.
Suppose $g$ is on $\Gamma_0(N)$ with character $\psi$, meaning $g(\gamma z)=\psi(\gamma)g(z)$ for all $\gamma \in \Gamma_0(N)$ and $z \in \mathbb{H}$. Write $\psi=\psi^{(Q)}\psi^{(R)}$, where $\psi^{(Q)}$ has modulus $Q$ and $\psi^{(R)}$ has modulus $R$. Let $h(z)=g(W_Qz) = g(W_Q' z)$ (the latter equality following from Remark \ref{remark:AtkinLehnerRelation}). Then $h(\gamma z) = g(W_Q \gamma z) = g(\gamma' W_Q' z) = \psi(\gamma') h(z)$.  By Lemma \ref{a-lprop}, $\psi(\gamma') = \psi'(\gamma)$, where
\begin{equation}
\label{eq:psi'def}
\psi':=\overline{\psi}^{(Q)}\psi^{(R)}.
\end{equation}

\subsection{Eisenstein Series} \label{background}
\indent  Let $\chi_1, \chi_2$ be primitive Dirichlet characters modulo $q_1,q_2$, respectively, such that
$\chi_1 \chi_2(-1)= 1$.  By convention, we allow $q_i=1$, in which case $\chi_i(n)=1$ for all $n \in \mathbb{Z}$. The \emph{newform Eisenstein series} attached to $\chi_1, \chi_2$ is defined as
$$E_{\chi_1,\chi_2}(z,s)=\frac{1}{2}\sum\limits_{(c,d)=1}\frac{(q_2y)^{s}\chi_1(c)\chi_2(d)}{{|cq_2z+d |}^{2s}}, \hspace{1em} \text{Re}(s)>1. $$ 
\noindent The function
 $E_{\chi_1,\chi_2}$ is automorphic on $\Gamma_0(q_1 q_2)$ with character $\psi = \chi_1 \overline{\chi_2}$, and we refer the reader to \cite{young} as a convenient reference for this fact and others to follow. 

Define the ``completed" Eisenstein series by
\begin{equation}\label{completed}
     E_{\chi_1,\chi_2}^*(z,s)=\frac{{(q_2/\pi)}^s}{\tau(\chi_2)}\Gamma(s)L(2s,\chi_1\chi_2)E_{\chi_1,\chi_2}(z,s),
\end{equation}
\noindent where $\tau(\chi_2)$ denotes the Gauss sum, and $L(s, \chi_1 \chi_2) = \sum \limits_{n \geq 1} \chi_1 \chi_2(n) n^{-s}$. The Fourier expansion takes the form
\begin{equation} 
\label{fourier_expansion}
E_{\chi_1, \chi_2}^*(z,s) = 
e_{\chi_1, \chi_2}^*(y,s) + 
2\sqrt{y}\sum_{n \neq 0}\lambda_{\chi_1, \chi_2} (n,s) e(nx) K_{s-\frac{1}{2}} ( 2 \pi |n| y),
\end{equation}
 where $K_{\nu}$ is the $K$-Bessel function,
\begin{equation}
\label{eq:Echi1chi2FourierCoefficient}
 \lambda_{\chi_1, \chi_2} (n,s) = \chi_2 (\textnormal{sgn} (n) ) \sum_{ab=|n|}\chi_1 (a) \overline{\chi_2} (b) (b/a)^{s-\frac12},
\end{equation}
and the constant term $e_{\chi_1, \chi_2}^*(y,s)$ is given by
\begin{equation}
\delta_{q_1=1} q_2^{2s} \frac{\pi^{-s}}{\tau(\chi_2)} \Gamma(s) L(2s, \chi_2) y^s
+
\delta_{q_2=1} q_1^{2-2s} \frac{\pi^{-(1-s)}}{\tau(\overline{\chi_1})} \Gamma(1-s) L(2-2s,\overline{\chi_1}) y^{1-s}.
\end{equation}
Throughout the remainder of the paper, we assume that $q_1 q_2 > 1$.  Combined with \eqref{fourier_expansion}, this condition ensures that $e_{\chi_1, \chi_2}^*(y,s)$ and $E_{\chi_1, \chi_2}^*(z,s)$ are analytic for all $s \in \mathbb{C}$.
With some simplifications (see \cite[(1.5), (1.6)]{2019}), \eqref{fourier_expansion} specializes as 
\begin{equation}\label{kronecker}
E^*_{\chi_1,\chi_2}(z,1)= F_{\chi_1,\chi_2}(z) + \chi_2 (-1)\overline{F}_{\overline{\chi_1}, \overline{\chi_2}} (z),
\end{equation} 
where 
\begin{equation}
\label{eq:fchi1chi2FourierExpansion}
 F_{\chi_1,\chi_2}(z)=
 c_1 z + c_0 + 
 \sum_{n =1}^{\infty} \frac{\lambda_{\chi_1, \chi_2}(n,1)}{\sqrt{n}}  e(nz)
\end{equation}
and where
\begin{equation}
c_1 = \delta_{q_1=1} \frac{q_2^2 L(2, \chi_2)}{2 \pi i \tau(\chi_2)} = \delta_{q_1=1} \pi i L(-1, \overline{\chi_2}),
\qquad
\text{and}
\qquad
c_0 =  \delta_{q_2=1} \tfrac12 L(1,\chi_1).
\end{equation}
Note that the calculation of both constants $c_1$ and $c_0$ used the functional equation of the Dirichlet $L$-function. 

A key feature of the newform Eisenstein series is that they 
are pseudo-eigenfunctions of the  Atkin-Lehner operators.
By \cite[(9.3)]{young} (see also \cite{weisinger}), 
\begin{equation}\label{conversionC}
E_{\chi_1,\chi_2}(W_Q z, s) = C E_{\chi'_1, \chi'_2}(z,s),
\end{equation}
where $C = \chi_1^{(Q)}(-1) \psi^{(Q)}(q_1^{(R)} u_0) \overline{\psi}^{(R)}(q_2^{(Q)} r_0)$.
This evaluation of $C$ can be found in \cite[Section 9.1]{young}
with $r_0,u_0=1$, but it is easy to extend the calculation for general $r_0$ and $u_0$. 
Note that $\psi'$ defined by \eqref{eq:psi'def} equals $\chi_1' \overline{\chi_2}'$.
Similarly, for the completed Eisenstein series, we derive
\begin{equation} \label{completeEisen}
    E^{*}_{\chi_1,\chi_2}(W_Qz,1) = \beta E^*_{\chi'_1,\chi'_2}(z),
    \qquad \text{where} \qquad \beta = \frac{q_2 \tau(\chi'_2)}{q'_2 \tau(\chi_2)} C.
\end{equation}

\subsection{Dedekind Sums} \label{dedesums}
We now turn to the newform Dedekind sums that are constructed with the newform Eisenstein series.  
These Dedekind sums were defined in \cite{2019}, though with the assumption that $q_1 \neq 1$ \emph{and} $q_2 \neq 1$, in which case \eqref{dede} was derived after some extensive calculations.  We will extend their definition to cover the cases where at most one of the $q_i$ is $1$ (note that when $q_1 = q_2 = 1$, then this is the classical Dedekind sum, so this assumption is not restrictive).  Some of the proofs from \cite{2019} carry over nearly verbatim, in which case we will omit the details here and refer the reader to \cite{2019}.

For $\gamma \in \Gamma_0(N)$, define the function
\begin{equation} \label{phidef}
\phi_{\chi_1, \chi_2}(\gamma, z) = F_{\chi_1, \chi_2}(\gamma z) - \psi(\gamma) F_{\chi_1, \chi_2}(z),
\end{equation}
 where recall $\psi = \chi_1 \overline{\chi_2}$ and $\chi_1 \chi_2(-1) = 1$.
  This definition of $\phi_{\chi_1,\chi_2}$ extends that of \cite{2019} by allowing at most one of $c_1$ or $c_0$ to be nonzero. 
  The function $\phi_{\chi_1, \chi_2}(\gamma, z)$ is constant in terms of $z$; see
  \cite[Lemma 2.1]{2019} for a proof that applies to this extended definition.  Therefore we write $\phi_{\chi_1, \chi_2}(\gamma, z)$ more simply as $\phi_{\chi_1, \chi_2}(\gamma)$.
 As in \cite{2019}, we define the newform Dedekind sum by \begin{equation}\label{normalizing}
S_{\chi_1,\chi_2}(\gamma) = \frac{\tau(\overline{\chi_1})}{\pi i} \phi_{\chi_1,\chi_2}(\gamma).
\end{equation} 

The most important property of the newform Dedekind sum is that it is a group homomorphism on $\Gamma_1(N)$.  Precisely, we have
\begin{lemma}[\cite{2019} Lemma 2.2]\label{grouphomom}
Let $\gamma_1,\gamma_2 \in \Gamma_0(N).$ Then $$S_{\chi_1,\chi_2}(\gamma_1\gamma_2)=S_{\chi_1,\chi_2}(\gamma_1)+\psi(\gamma_1)S_{\chi_1,\chi_2}(\gamma_2).$$
\end{lemma}
\noindent The proof in \cite{2019} carries over identically to the case where $q_1 = 1$ or $q_2 = 1$.

 Note that $S_{\chi_1,\chi_2}(\gamma) = 0$ if and only if $F_{\chi_1, \chi_2}(\gamma z) = \psi(\gamma) F_{\chi_1, \chi_2}(z)$ for all $z \in \mathbb{H}$. So, the elements of $K_{\chi_1,\chi_2}$ are those for which $F_{\chi_1,\chi_2}$ transforms like an automorphic form.  This is some motivation for studying the kernel of the Dedekind sum.

\begin{proposition}\label{FirstColumn}
Suppose $\gamma_1,\gamma_2 \in \Gamma_0(N)$ have the same left column and suppose $q_1 \neq 1$. Then $S_{\chi_1,\chi_2}(\gamma_1) =  S_{\chi_1,\chi_2}(\gamma_2)$.
\end{proposition}
\begin{proof}
It follows from our conditions on $\gamma_1,\gamma_2$ that 
$\gamma_1 = \gamma_2 \omega$ with $\omega = (\begin{smallmatrix} 1 & b \\ 0 & 1 \end{smallmatrix})$ for some $b \in \mathbb{Z}$. Now, by Lemma \ref{grouphomom}, $S_{\chi_1,\chi_2}(\gamma_1) = S_{\chi_1,\chi_2}(\gamma_2 \omega) = S_{\chi_1,\chi_2}(\gamma_2)+\psi(\gamma_2)S_{\chi_1,\chi_2}(\omega).$ By \eqref{phidef} and \eqref{eq:fchi1chi2FourierExpansion},  if $q_1 \neq 1$ then $S_{\chi_1, \chi_2}(\omega) = 0$.
\end{proof}
Proposition \ref{FirstColumn} justifies our earlier notation $S_{\chi_1, \chi_2}(\begin{smallmatrix} a & b \\ c & d \end{smallmatrix}) = S_{\chi_1, \chi_2}(a,c)$, provided $q_1 \neq 1$. 
On the other hand, we have
\begin{equation}
\label{eq:DedSumUpperTriangularchi1trivial}
S_{1, \chi_2}(\begin{smallmatrix} 1 & b \\ 0 & 1 \end{smallmatrix}) = b L(-1, \chi_2),
\end{equation}
showing the condition $q_1 \neq 1$ is necessary in Proposition \ref{FirstColumn}.
Now define
\begin{equation}\label{phiWQ}
    \phi_{\chi_1,\chi_2}(W_Q, z)=\phi_{\chi_1,\chi_2}(W_Q)=F_{\chi_1,\chi_2}(W_Q z) - \beta F_{\chi'_1,\chi'_2}(z)
\end{equation} 
for $z \in \mathbb{H}$. Note, $\phi_{\chi_1,\chi_2}(W_Q)$ may depend on the choice of $r,t,u,$ and $v$ in $W_Q$, but $\beta$ does not (though $\beta$ does depend on $u_0$, $r_0$, the characters $\chi_i$, and so on).
\begin{lemma} 
\label{lemma:DedSumAtALconstant}
The function $\phi_{\chi_1,\chi_2}(W_Q,z)$ is independent of $z$.
\end{lemma}
\begin{proof}
 From \eqref{kronecker} and \eqref{completeEisen}, it immediately follows that $$\phi_{\chi_1,\chi_2}(W_Q, z) = - \chi_2 (-1) \overline{\phi}_{\overline{\chi_1},\overline{\chi_2}}(W_Q z).$$ Since $\phi_{\chi_1,\chi_2}(W_Q, z)$ is holomorphic and $\overline{\phi}_{\overline{\chi_1},\overline{\chi_2}}(W_Q, z)$ is anti-holomorphic, $\phi_{\chi_1,\chi_2}(W_Q, z)$ must be constant in $z$. 
\end{proof}
Lemma \ref{lemma:DedSumAtALconstant} justifies writing $\phi_{\chi_1,\chi_2}(W_Q, z) = \phi_{\chi_1,\chi_2}(W_Q)$.
Analogously to (\ref{normalizing}), we define the Dedekind sum $S_{\chi_1,\chi_2}$ associated to $W_Q$ as 
\begin{equation} \label{dedeWQ}
    S_{\chi_1,\chi_2}(W_Q) = \frac{\tau (\overline{\chi_1})}{\pi i} \phi_{\chi_1,\chi_2}(W_Q).
   \end{equation}

\section{Proof of the generalized reciprocity formula} \label{reciprocity}
\begin{theorem}[Generalized Reciprocity Formula]\label{rec}
Let $\chi_1,\chi_2$ be primitive Dirichlet characters with moduli $q_1,q_2$, respectively, such that $q_1 q_2 > 1$, and $\chi_1 \chi_2(-1) = 1$.  Let $N = q_1 q_2$ and let $W_Q$ and $W_Q'$ be Atkin-Lehner operators as described in Section \ref{atkin}. Then 
for $\gamma, \gamma' \in \Gamma_0(N)$ related by
$W_Q \gamma = \gamma' W_Q'$, we have the following reciprocity formula:
\begin{equation}
\label{eq:ALreciprocityGeneral}
S_{\chi_1,\chi_2}(W_Q) + {\xi}S_{\chi'_1\chi'_2}(\gamma) =  \psi'(\gamma)S_{\chi_1,\chi_2}(W_Q') + S_{\chi_1,\chi_2}(\gamma'),
\end{equation}
\noindent where $\xi = \frac{\tau (\overline{\chi_1})}{\pi i}\beta$ and $\psi' = \chi_1'\overline{\chi_2'}.$ 
\end{theorem}
\begin{proof}
Recall (\ref{phidef}) and (\ref{phiWQ}). Now we calculate
\begin{equation}\label{4}
F_{\chi_1,\chi_2}(W_Q \gamma z) - \beta \psi'(\gamma) F_{\chi'_1,\chi'_2}(z)
\end{equation}
in two ways. First, note that it equals  $$ \underbrace{F_{\chi_1,\chi_2}(W_Q \gamma z) - \beta F_{\chi'_1,\chi'_2}(\gamma z)}_{\phi_{\chi_1,\chi_2}(W_Q)} + \beta \underbrace{(F_{\chi'_1,\chi'_2}(\gamma z) - \psi'(\gamma) F_{\chi_1',\chi_2'}(z))}_{ \phi_{\chi_1',\chi_2'}(\gamma)}.$$
\noindent Alternatively, we use $W_Q \gamma = \gamma' W_Q'$ to see that (\ref{4}) equals $$F_{\chi_1,\chi_2}(\gamma' W_Q' z) - \psi'(\gamma)F_{\chi_1,\chi_2}(W'_Q z) + \psi'(\gamma)(F_{\chi_1,\chi_2}(W'_Q z) - \beta F_{\chi'_1,\chi'_2}( z)).$$ 
Using $\psi' (\gamma) = \psi(\gamma')$, this becomes $$\underbrace{F_{\chi_1,\chi_2}(\gamma' W_Q' z) - \psi(\gamma')F_{\chi_1,\chi_2}(W'_Q z)}_{\phi_{\chi_1,\chi_2}(\gamma ')} + \psi'(\gamma)\underbrace{(F_{\chi_1,\chi_2}(W'_Q z) - \beta F_{\chi'_1,\chi'_2}( z))}_{\phi_{\chi_1,\chi_2}(W_Q')}.$$
Since the value of $\beta$ depends solely on $r_0$ and $u_0$ and not $r,t,u$, and $v$, it remains unchanged between the formulas involving $W_Q$ and $W_Q'$. Equating the two expressions for \eqref{4}, we infer
$$\phi_{\chi_1,\chi_2}(W_Q) + \beta \phi_{\chi_1',\chi_2'}(\gamma) = \phi_{\chi_1,\chi_2}(\gamma ') + \psi'(\gamma) \phi_{\chi_1,\chi_2}(W_Q').$$ \noindent 
 By \eqref{normalizing} and \eqref{dedeWQ}, we deduce \eqref{eq:ALreciprocityGeneral}.
\end{proof}

\section{Kernel} \label{kernel}




In this section, we use
Theorem \ref{rec} to further the study initiated in \cite{2020} of the kernel of Dedekind sums.
We begin with a proof of Theorem \ref{kerthm1}.
\begin{proof}  
Given $r,u$ as stated in the theorem, let
$W_Q = (\begin{smallmatrix} Qr & t \\ Nu & Qv \end{smallmatrix})$ be an Atkin-Lehner operator (the conditions $(r,R) =1$ and $(u,Q)=1$ are sufficient to ensure such a matrix exists, and then $r_0, u_0$ are determined). 

Let 
 $\gamma=
(\begin{smallmatrix}
1 & -k  \\
0 & 1 \end{smallmatrix})$.  
Recall $\gamma'$ is defined by $W_Q \gamma = \gamma' W_Q$ and is calculated in \eqref{eq:bruteforce}, giving
\begin{equation}
\label{eq:gamma'formulaKernelSection}
 \gamma' = \begin{pmatrix}
1+Nkur & * \\
NRku^2 & *
\end{pmatrix}.
\end{equation}
The strategy for the proof is now to show that the assumptions in Theorem \ref{kerthm1} imply that 
$S_{\chi_1', \chi_2'}(\gamma) = 0$,
and so the reciprocity formula Theorem \ref{rec} implies 
$S_{\chi_1, \chi_2}(\gamma') = 0$.  The condition $q_1 \neq 1$ is required for $S_{\chi_1, \chi_2}(\gamma')$ to only depend on the left column of $\gamma'$.

We now show that the condition $q_2 \neq R$ implies that $q_1' \neq 1$.
Recall that $q_1' = (q_2, Q)(q_1, R)$, so that $q_1' = 1$ implies $(q_2, Q) =1$ and $(q_1, R) =1$.  However, these latter two conditions imply $q_2 = R$ (and also $q_1 = Q$), a contradiction.

Using $q_1' \neq 1$, and by \eqref{phidef} and \eqref{eq:fchi1chi2FourierExpansion}, we conclude 
$S_{\chi_1', \chi_2'}(\gamma) = 0$.
Therefore,
$S_{\chi_1, \chi_2}(\gamma') = 0$,
by the reciprocity formula.  That is, $S_{\chi_1, \chi_2}(1+Nkur, NRku^2) = 0$, where we have used $q_1 \neq 1$ in order to invoke Proposition \ref{FirstColumn}.
 
 Finally, by taking $\gamma=
(\begin{smallmatrix}
-1 & -k \\
0 & -1
\end{smallmatrix})$, we obtain 
\begin{equation} \label{-1}
     S_{\chi_1,\chi_2}(-1+Nkur,NRku^2)=0. \qedhere
\end{equation}
\end{proof}

In the following example, we elaborate on a special case of Theorem \ref{kerthm1} to demonstrate its versatility.
\begin{exmp}
Consider $K_{7,11}$ displayed in Figure \ref{k711}. It follows from Theorem \ref{kerthm1} (with 
$R=7$, $Q=11$, 
$u = 1$, and $(r, 7) = 1$) that $(\pm 1 + 77kr, 539k)$ is in the kernel.  This encompasses all elements with $c = 539k$, which are those circled in Figure \ref{k711}  (note that due to the $\pm1$ term in Theorem \ref{kerthm1}, each circle contains a pair of kernel elements with close upper-left entries).
\end{exmp}
\begin{figure}[H]
\centering
\includegraphics[width=3in]{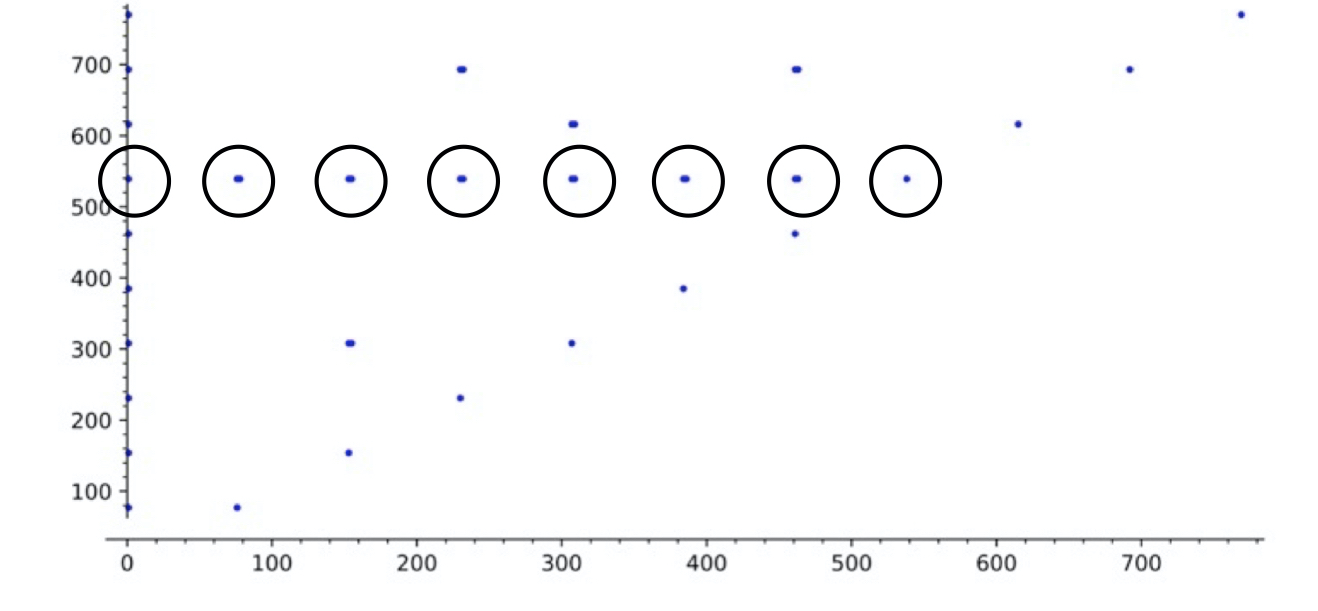}
\caption{Elements of $K_{7,11}$ for $1 \leq c \leq 10q_1q_2$} \label{k711}
\end{figure}

 Note that in Figures \ref{figureIntro2020elements} and \ref{k711}, the only kernel elements present in rows $c = Nkq_2$ are those equivalent to $(\pm 1, Nkq_2)$. We refer the reader to \cite{2020} for additional examples of this phenomenon. The following proposition gives some evidence for this sparseness by demonstrating that no element of the form $(a,c) = (\pm 1 + Nkr, Nkq_2)$ with nonzero $k$ is in the kernel. 
\begin{proposition}
Let $q_1 q_2 = N$ with $(q_1,q_2)=1$, $q_1 \neq 1$. Suppose $r \in \mathbb{Z}$ with $(r,q_2)=1$.  Then for any nonzero $k \in \mathbb{Z}$, we have  
\begin{equation}
\label{eq:DedekindSumDoesntVanish}
S_{\chi_1,\chi_2}(\pm 1 + Nk r, Nk q_2) \neq 0.
\end{equation}
\end{proposition}
\begin{proof}
The proof follows the same framework as Theorem \ref{kerthm1}.  Let $W_Q$ be an Atkin-Lehner operator with
$Q = q_1$ and $R = q_2$, so $(q_1,R)=1$ and $(q_2,Q)=1$. The effect of $W_Q$ on $q_1,q_2$ is such that $q'_1 = 1$.  Let $\gamma = (\begin{smallmatrix} 1 & -k \\ 0 & 1 \end{smallmatrix})$. By \eqref{eq:DedSumUpperTriangularchi1trivial}, $S_{1, \chi_2'}(\gamma) = -k L(-1, \overline{\chi_2}') \neq 0$.
Therefore, with $\gamma'$ as in \eqref{eq:gamma'formulaKernelSection}, we have $S_{\chi_1, \chi_2}(\gamma') \neq 0$. This shows the claim \eqref{eq:DedekindSumDoesntVanish}.
\end{proof}

\section{Acknowledgements}
This research was conducted in summer 2021 during the Texas A\&M University REU.  We thank the teaching assistants Agniva Dasgupta, Joshua Goldstein, and Zhengye Zhou for their support during the REU.  We  especially thank Evuilynn Nguyen and Juan J. Ramirez for the creation of the graphs seen throughout. Finally, we thank the Department of Mathematics at Texas A\&M and the NSF (DMS-1757872) for supporting the REU. This material is based upon work supported by the National Science Foundation under agreement No. DMS-2001306 (M.Y.). Any opinions, findings, and conclusions or recommendations expressed in this material are those of the authors and do not necessarily reflect the views of the National Science Foundation.

\bibliographystyle{alpha}
\bibliography{ref-bib}

\end{document}